\documentclass[a4paper, 12pt]{article}
\usepackage{tgheros}
\usepackage{amsthm}
\newtheorem{thm}{Theorem}
\newtheorem{cor}[thm]{Corollary}
\newtheorem{conj}{Conjecture}
\newtheorem{lemma}[thm]{Lemma}

\usepackage{amssymb,amsmath}

\newtheorem*{theorem*}{Theorem}

\DeclareMathOperator{\F}{\mathbb{F}}

\begin{document}
\baselineskip=16.3pt
\parskip=14pt

\begin{center}
\section*{Linearized Polynomials, Galois Groups and Symmetric Power Modules}

{\large 
Rod Gow and Gary McGuire
 \\ { \ }\\
UCD School of Mathematics and Statistics\\
University College Dublin\\
Ireland}
\end{center}

 \subsection*{Abstract}
 
 We investigate some Galois groups of linearized polynomials  
  over fields such as $\F_q(t)$. The space of roots of such a polynomial is a module for its Galois group.
  We present a realization of the symmetric powers of this module, 
  as a subspace of  the splitting field
  of another linearized polynomial.

\section{Introduction}

Let $p$ be a prime number, and let $q=p^a$ where $a$ is a positive integer.
Let $\F_q$ denote the finite field with $q$ elements.
Let $F$ be a field of characteristic $p$,
and assume that $F$ contains $\F_q$.
A $q$-linearized polynomial over $F$  is a polynomial of the form 
\begin{equation}\label{lin1}
L=a_0x+a_1x^q+a_2x^{q^2}+\cdots+a_n x^{q^n} \in F[x].
\end{equation}
If $a_n\not=0$ we say that $n$ is the $q$-degree of $f$.
We will usually say $q$-polynomial instead of $q$-linearized polynomial.

The set of roots of a $q$-polynomial $L$ forms an $\F_q$-vector space,
which is contained in a splitting field of $L$. We make this statement more precise in the following 
simple lemma, which also serves as an introduction to the topics of this paper.

\begin{lemma} \label{vector_space}
Let $F$ be a field of prime characteristic $p$ that contains $\F_q$. Let $L$ be a $q$-polynomial of $q$-degree $n$ in $F[x]$, with
\[
L=a_0x+a_1x^q+a_2x^{q^2}+\cdots+a_n x^{q^n}.
\]
Let $E$ be a splitting field for $L$ over $F$ and let $V$ be the set of roots of $L$ in $E$. Let $G$ be the Galois group of $E$
over $F$. Suppose that $a_0\neq 0$. Then $V$ is an $\F_q$-vector space of dimension $n$ and $G$ is naturally a subgroup
of $GL(n,q)$.
\end{lemma}

\begin{proof}
The derivative $L'(x)=a_0$ and is thus a nonzero constant under the hypothesis above. It follows that $L$ has no repeated
roots and thus $|V|=q^n$. Let $\alpha$ and $\beta$ be elements of $V$. Then $L(\alpha)=L(\beta)=0$. Since
\[
(\alpha+\beta)^{q^i}=\alpha^{q^i}+\beta^{q^i}
\]
for all $i\geq 0$, it is clear that $L(\alpha+\beta)=0$, and hence $\alpha+\beta\in V$.

Let $\lambda$ be an element of $\F_q$. Since $\lambda^q=\lambda$, we have $(\lambda\alpha)^{q^i}=\lambda \alpha^{q^i}$. It
follows that $L(\lambda\alpha)=\lambda L(\alpha)=0$ and hence $\lambda\alpha\in V$. These arguments show that $V$ is a vector space
over $\F_q$. Furthermore, since $|V|=q^n$, $V$ has dimension $n$ over $\F_q$.

Let $\sigma$ be an element of $G$. By definition of Galois group action, $G$ permutes the roots of $L$ and hence maps $V$ into itself.
In addition, since $G$ is a group of field automorphisms of $E$, we have
\[
\sigma(\alpha+\beta)=\sigma(\alpha)+\sigma(\beta)
\]
for all $\alpha$ and $\beta$ in $V$. 

Now let $\lambda$ be an element of $\F_q$. Since $\F_q$ is assumed to be a subfield of $F$ and $G$ fixes $F$ elementwise, we have
$\sigma(\lambda)=\lambda$. Then, again by definition of Galois group action,
\[
\sigma(\lambda\alpha)=\sigma(\lambda)\sigma(\alpha)=\lambda\sigma(\alpha).
\]
This shows that the action of $G$ on $V$ is $\F_q$-linear and thus $G$ may be considered to be a subgroup of $GL(n,q)$.
\end{proof}

We call $V$ the ($\F_q$-vector) space of roots of $L$.
The proof shows that $V$ is an $FG$-module.

We remark that the Galois group of a generic $q$-polynomial is $GL(n,q)$,
as shown for example in \cite{W} when proving a theorem of Dickson.

We recall the definition of the symmetric powers.
Let $V$ be any vector space over a field $F$.
The $r$-th symmetric power of $V$, $Sym^r(V)$,  is another vector space constructed 
as the quotient space of $V^{\otimes r}$ by the subspace generated by all
$v_1\otimes \cdots \otimes v_r - v_{\sigma(1)}\otimes \cdots \otimes v_{\sigma(r)}$
where $\sigma \in S_r$.
If $e_1, e_2, \ldots ,e_r$ is a basis for $V$ then
$e_1^{k_1} e_2^{k_2} \cdots e_r^{k_n}$
where $k_1+\cdots +k_n =r$ is a basis for  $Sym^r(V)$,
where we denote the operation as ordinary multiplication.
The $r$-th symmetric power can thus be identified with the space of all 
homogeneous polynomials of degree $r$.
The symmetric powers play an important role in  representation theory.

For the remainder of this section we shall outline the structure of this paper.
 First,  Section \ref{LPB} 
 presents a background discussion of $q$-polynomials,
with a focus on the relevant issues for this article.

%The motivation for this paper arose as follows.
Let $L(x)$ be a monic $q$-linearized polynomial with the same hypotheses 
as Lemma  \ref{vector_space}.
Let $M(x)=L(x)/x$.
Then it is easy to see that $M(x)$ is a monic polynomial in $x^{q-1}$,
say $M(x)=P(x^{q-1})$.
The polynomial $P(x)$ is monic of degree $(q^n-1)/(q-1)$ and is called the
projective polynomial associated to $L$.
If $M$ is irreducible over $F$ then so is $P$.

We will show in Section \ref{SLT}
that any polynomial divides a linearized polynomial.
Applying this lemma to $P(x)$, let $L_P(x)$ be the linearized polynomial of minimal degree
that is divisible by $P(x)$. 
We will show in Section \ref{SLT}
that $L_P$ has degree $q^d$ where $d$ is the dimension of the $\F_q$ span of the roots of $P$.
So we have a construction that starts with $L$, 
then constructs $P$, and then another linearized polynomial $L_P$.

We wish to compare the Galois groups of $L$ and $P$,
which we denote by $G_L$ and $G_P$ respectively.
Let $K_L$ be a splitting field for $L$.
Let $K_P$ be a splitting field for $P$, which is the same as a splitting field for $L_P$
(so $P$ and $L_P$ have the same Galois group).
If $\alpha\in K_L$ is a nonzero root of $L$, then 
$P(\alpha^{q-1})=0$ and so $L_P(\alpha^{q-1})=0$.
Each $\alpha$ and all its $\F_q^*$-multiples give rise to the same root of $P$.
Conversely, if $\beta$ is a root of $P$ then the roots of
$x^{q-1}-\beta$ are roots of $L$.
This implies that $K_P \subseteq K_L$ and that $K_P$ 
contains all the $(q-1)$-th powers of the roots of $L$.
This also implies that $G_L=Gal(L)$ has a normal subgroup $N=Gal(K_L:K_P)$ 
such that $G_L/N \cong G_P$, and that $N$ will be a subgroup of a cyclic
group of order $q-1$.

The result we wish to present is that the
symmetric powers are easily visible in this setting. 
To summarize our results, let $V_L$ denote 
the space of roots of $L$, a $G_L$-module,
and let $V_P$ denote the space of roots of $L_P$, a $G_P$-module.
By the construction, since the $(q-1)$-th powers
of the roots of $L$ are roots of $L_P$,
one might intuitively expect the $(q-1)$-th symmetric power of $V_L$ to 
be related to $V_P$.
We will see that in order to make this precise, the following idea is important.

Let $\alpha_1, \ldots , \alpha_n$ be a basis for $V_L$. 
The linear independence of the elements
$\alpha_1^{k_1}\alpha_2^{k_2}\cdots \alpha_n^{k_n}$ 
%where $k_1+\cdots +k_n =p-1$ (and the $k_j$ are nonnegative integers) 
will be of crucial importance when trying to find symmetric powers.
Thus, we will study the evaluation mapping
$x_1^{k_1} x_2^{k_2} \cdots x_n^{k_n}\mapsto
\alpha_1^{k_1}\alpha_2^{k_2}\cdots \alpha_n^{k_n}$
from the space  of all homogeneous polynomials in $n$ variables,
having degree $r$ and coefficients in $F$, to $E$.
In order to realize a copy of  the $r$-th symmetric power in $E$, 
it is necessary and sufficient that 
this evaluation map be injective.

Sections \ref{CSR}-\ref{inj_sect} of the paper present the details of the above summary.
Section  \ref{CSR} begins the discussion of the injectivity of evaluation maps.
Next in Section \ref{proj_symm} we present the homomorphism from some
symmetric powers to $E$, and explain the importance of injectivity.
Section \ref{inj_sect} has a detailed discussion of the injectivity of the evaluation map,
 and proves the injectivity of the evaluation map in some cases.
 Section \ref{qdeg2} discusses the particular case of $q$-degree 2.
In Section \ref{RTM}  we talk about $V_L$ as a module for the Galois group,
and discuss the irreducibility of the symmetric power modules.

\section{Linearized Polynomials Background}\label{LPB}

We present two simple results which we will need later.
The first is a converse to Lemma \ref{vector_space}.
The proof is essentially the same as Theorem 3.52  in \cite{LN},
which goes back to Dickson.

\begin{lemma}\label{lin}
Let $V$ be a finite dimensional vector space over $\F_q$, which is contained in a field extension 
$E$ of $\F_q$.
Then the polynomial $\prod_{v\in V} (x-v)$ is a $q$-linearized polynomial. 
\end{lemma}

\begin{proof}
Let $\alpha_1, \ldots, \alpha_n \in E$ be a basis for $V$.
Consider the polynomial in $E[x]$
\[
D(x):=\det
 \left[ \begin{array}{ccccc}
   \alpha_1&\alpha_1^q&\alpha_1^{q^2}&\cdots&\alpha_1^{q^{n}}\\
 \alpha_2&\alpha_2^q&\alpha_2^{q^2}&\cdots&\alpha_2^{q^{n}}\\
 \vdots&\vdots&\vdots&\ddots&\vdots\\
  \alpha_n&\alpha_n^q&\alpha_n^{q^2}&\cdots&\alpha_n^{q^{n}}\\
  x&x^q&x^{q^2}&\cdots&x^{q^{n}}\\
 \end{array} \right]
\]
which  will clearly be a $q$-polynomial in $x$, of $q$-degree at most $n$.
We claim that the roots of $D(x)$ are precisely the elements of $V$,
 from which it follows that $D(x)$ is a scalar multiple of  $\prod_{v\in V} (x-v)$.
 
 The claim follows by observing that each $\alpha_i$ is a root of $D(x)$,
 and since $D(x)$ is a $q$-polynomial, all $\F_q$-linear combinations of
 the $\alpha_i$  are roots of $D(x)$ by Lemma  \ref{vector_space}.
 Since $V$ has $q^n$ elements, and $D(x)$ has degree at most $q^n$,
 it follows that $D(x)$ has degree exactly $q^n$ and the proof is complete.
\end{proof}

We next consider a slight variation on the theme of $q$-polynomials. Let $s\geq 1$ be an integer and let
\[
L(x)=a_0x+a_1x^{q^s}+a_2x^{q^{2s}}+\cdots+a_m x^{q^{ms}} \in F[x],
\]
where we assume that $a_m\neq 0$. Then $L$ is a $q$-polynomial of $q$-degree $ms$ and it is also a $q^s$-polynomial. We will
not always require that $\mathbb{F}_{q^s}$ is a subfield of $F$ when $s>1$.

The following result generalizes Lemma  \ref{vector_space}.

\begin{lemma} \label{semi_linear_automorphism}
Let $F$ be a field of prime characteristic $p$ that contains $\F_q$. Let $L$ be a $q^s$-polynomial of $q^s$-degree $m$ in $F[x]$. Assume
that the coefficient of $x$ in $L$ is nonzero.
Let $E$ be a splitting field for $L$ over $F$ and let $V$ be the set of roots of $L$ in $E$. Let $G$ be the Galois group of $E$
over $F$.  
\begin{enumerate}
\item The field $\mathbb{F}_{q^s}$ is a subfield of $E$.
\item $V$ is an $m$-dimensional vector space over $\mathbb{F}_{q^s}$.
\item $G$ acts on $V$ as a group of automorphisms that are semilinear
with respect to the group of $\F_q$-automorphisms of $\mathbb{F}_{q^s}$ induced by the $q$-th power map.
\item $G$ contains a normal subgroup $H$, say, such that $H$ is a subgroup of $GL(m,q^s)$ and $G/H$ is cyclic of order dividing $s$.
\end{enumerate}
\end{lemma}

\begin{proof}
It is clear from the proof of Lemma \ref{vector_space} that if $\alpha$ and $\beta$ are in $V$, so also is $\alpha+\beta$. 
Working in the algebraic closure of $F$, let $\lambda$ be an element of $\mathbb{F}_{q^s}$. Then $\lambda^{q^s}=\lambda$ and it
follows that
\[
L(\lambda\alpha)=\lambda L(\alpha)=0
\]
for all $\alpha\in V$. Thus $\lambda\alpha\in V$. Since $E$ contains $(\lambda\alpha)\alpha^{-1}=\lambda$, it follows that
$\lambda\in E$ and thus $E$ contains a copy of $\mathbb{F}_{q^s}$. Since $L$ has $q^{ms}$ different roots under the hypothesis above,
$V$ has dimension $m$ over $\mathbb{F}_{q^s}$.

The copy of $\mathbb{F}_{q^s}$ contained in $E$ is a normal subfield, since it is the splitting field of $x^{q^s}-x$ over $\F_q$ (and $F$ contains 
$\F_q$). It follows that $G$ maps  $\mathbb{F}_{q^s}$ into itself and induces a subgroup of $\F_q$-automorphisms of the field.

Following the proof of Lemma \ref{vector_space}, $G$ maps $V$ into itself and satisfies
\[
\sigma(\alpha+\beta)=\sigma(\alpha)+\sigma(\beta)
\]
for all $\sigma\in G$ and all $\alpha$ and $\beta$ in $V$. Let $\lambda$ be an element of $\mathbb{F}_{q^s}$
and $\sigma$ be an element of $G$. Then we have
\[
\sigma(\lambda\alpha)=\sigma(\lambda)\sigma(\alpha)
\]
for all $\alpha$ in $V$. This implies that $G$ acts semilinearly on $V$ with respect to the group of automorphisms it induces
of $\mathbb{F}_{q^s}$.

Finally, let $H$ be the subgroup of $G$ that acts trivially on $\mathbb{F}_{q^s}$. Standard Galois theory shows that
$H$ is normal in $G$ and $G/H$ is isomorphic to a subgroup of $\F_q$-automorphisms of $\mathbb{F}_{q^s}$. Since
the Galois group of $\mathbb{F}_{q^s}$ over $\F_q$ is cyclic of order $s$, the quotient $G/H$ is cyclic of order dividing
$s$.
\end{proof}

We remark that Lemma \ref{semi_linear_automorphism} is really only of interest when $\mathbb{F}_{q^s}$ is not a subfield
 of $F$, since otherwise the result is a repetition of Lemma \ref{vector_space}.

\section{Serre's Linearization Trick}\label{SLT}

In this paper we will sometimes  assume that $L(x)/x$ is irreducible,
which is the generic case, and will usually coincide with the Galois group being $GL(n,q)$.
In this section we will approach from a different direction, and we will 
construct $q$-polynomials $L$ such that $L(x)/x$ is not irreducible.
One can hope for more exotic Galois groups in such cases, which is indeed
the purpose of our approach. 

Given any polynomial $f\in F[x]$, it is  useful to find a $q$-polynomial
in $F[x]$ that is divisible by $f$.
The following Lemma is well known, and can be found for example in Kedlaya \cite{K}.
We give two proofs here.

\begin{lemma}\label{mind}
Let $F$ be a field of characteristic $p$ that contains $\F_q$.
For any  polynomial $f\in F[x]$ of degree $m$,
there exists  $L\in F[x]$  which is divisible by $f$ and has  the form
\[
L(x)=\sum_{i=0}^d b_i x^{q^i}
\]
where $d\le m$. 
\end{lemma} 

\begin{proof}
First proof.
Consider the $m+1$ elements $x^{q^i}$ mod $f$, $0\le i \le m$, of the
$m$-dimensional $F$-vector space $F[x]/(f)$. 
There must be a nontrivial linear dependence relation between these elements, say
\[
\sum_{i=0}^m b_i ( x^{q^i} \ \text{mod} \ f)=0.
\]
This implies $(\sum_{i=0}^m b_i  x^{q^i}) \ \text{mod} \ f=0$,
and therefore  $L(x)=\sum_{i=0}^m b_i  x^{q^i}$ is the required  polynomial.

Second proof (this is essentially the proof in \cite{K}).
Let $E$ be a splitting field for $f$ and let
$\beta_1, \ldots , \beta_m$ be the roots of $f$ in $E$, which are distinct by hypothesis.
Let $V$ be the $\F_q$-span of $\beta_1, \ldots , \beta_m$.
The dimension of $V$ is at most $m$.
By Lemma \ref{lin} the polynomial 
$\prod_{v\in V} (x-v)$ is a $q$-linearized polynomial,
and it will clearly be divisible by $f$.
This polynomial has coefficients in $F$ because the coefficients are fixed by the Galois group.
\end{proof}

There may be a linear relation of smaller degree, 
which would result in an additive polynomial $L$ of smaller 
$q$-degree than $m$.
Abhyankar and Yie \cite{AY} 
say that $f$ \underline{linearizes at $d$} if there exists a 
$q$-linearized polynomial
$L$ of $q$-degree $d$ such that $f$ divides $L$.
Lemma \ref{mind} says that $f$ will linearize at an integer $d$ with $d\le m$.
If $f$ linearizes at $d$ where $d$ is significantly smaller than $m$, then useful information 
can be obtained about the Galois group of $f$,
because the Galois group of $f$ is (almost always) equal to the Galois group of $L$,
which is isomorphic to a subgroup of $GL(d,q)$.
They refer to this as Serre's linearization trick.

Our next lemma focuses on the smallest $d$ such that $f$ linearizes at $d$.

\begin{lemma}\label{mind2}
Let $F$ be a field of characteristic $p$ that contains $\F_q$.
The minimal $q$-degree of a $q$-linearized 
polynomial $L\in F[x]$  which is divisible by $f$  is 
equal to the dimension of the $\F_q$-span of the roots of $f$ in a splitting field.
\end{lemma} 

\begin{proof}
Let $f$ be a polynomial of degree $m$, with coefficients in the field $F$.
Let $d$ be the smallest positive integer such that there exists a 
$q$-linearized polynomial $L\in F[x]$ of $q$-degree $d$ which is divisible by $f$.
By Lemma \ref{mind}, $d\le m$.

Let $E$ be a splitting field for $f$.  The set of all 
$\F_q$-linear combinations of the roots of $f$  is
an $\F_q$-vector space $V \subseteq E$. 
Let $L'$ be the monic polynomial whose  roots
are precisely the elements of $V$. 
Then $L'$ must be a linearized polynomial by Lemma \ref{lin}, and $f$ 
divides $L'$ because its roots are a subset of the roots of $L'$. 
Let $d'$ be the $q$-degree of $L'$. Then $d'\ge d$ by definition of $d$.
If $d'>d$ then there would exist a linearized polynomial $L$ of degree $d$
that contains the roots of $f$ among its roots. 
Since $L$ is linearized, it also contains every linear combination of the roots of $f$
among its roots. Then $L'$ would divide $L$, which contradicts $d'>d$.
Therefore $d'=d$.
\end{proof}

%If $f$ linearizes at $d$, and $f$ does not linearize at any integer smaller than $d$, we say that $f$ \underline{minimally linearizes at $d$}.

{\bf Example:} (from \cite{CMT})
Consider the polynomial $f=x^{24}+x+t$ over $\F_2(t)$.
Using Magma  we find 
that $f$  linearizes at 12 and we find
the following linearized polynomial divisible by $f$:
\[
L(x)=x^{4096} + (t^{24} + t)x^{2048} + t^{128}x^{1024} + (t^{88} + t^{65})x^{512} + t^{16}x^{32} +
\]
\[
    t^9x^{16} + (t^{40} + t^{17})x^8 + x^2 + (t^{24} + t)x.
\]
The degrees of the irreducible factors of $L$ are
1,
24,
276,
1771,
2024.

{\bf Remark:}
An affine polynomial is a polynomial of the form $L(x)+c$ where $L$ is a linearized polynomial and $c\in F$.
It is possible that $f$ divides an affine polynomial of smaller $q$-degree than the minimal
positive integer at which $f$ linearizes. 
In the example above, $f$ minimally  linearizes at 12, however there exists an affine polynomial
of 2-degree 11 which is divisible by $f$.
It is the product of $f$ and the degree 2024 irreducible factor.

\section{Constructions related to the Space of Roots}\label{CSR}

\noindent We continue with the notation and themes of the previous section. 
Let $L$ be a $q$-polynomial of $q$-degree $n$ in $F[x]$. We assume that $L$ has $q^n$ different roots. These form the space of roots 
$V$, which is an $n$-dimensional vector space
over $\F_q$. The Galois group $G$ of $L$ acts $\F_q$-linearly on $V$, under the assumption that $F$ contains $\F_q$.

Associated to the vector space $V$ are such vector spaces as the symmetric powers of $V$, which we wish to consider in this section.
Let $x_1$, \dots, $x_n$ be $n$ algebraically independent indeterminates over $\F_q$ and let $A(\F_q)=\F_q[x_1, \ldots, x_n]$ be the ring
of polynomials in the $x_i$. As is well known, the general linear group $GL(n,q)$ acts on $A(\F_q)$, and $A(\F_q)$ is an $\F_q$-module for
the group. We may explain this idea as follows.

We identify $GL(n,q)$ with the group of invertible $n\times n$ matrices over $\F_q$. Given an element $g$ in $GL(n,q)$, write
$g$ as an $n\times n$ matrix $(a_{ij})$, $1\leq i,j\leq n$. Then we set
\[
gx_i=a_{i1}x_1+\cdots +a_{in}x_n
\]
for $1\leq i\leq n$. This linear action is extended to powers of the $x_i$, so that, for example, $g$ sends $x_i^r$ to $(gx_i)^r$ for
each positive integer $r$. The action is extended to monomials in the $x_i$
in the obvious way. Finally we extend the action to polynomials: given $P\in A(\F_q)$ and $g$ in $GL(n,q)$, we define
$P^g$ by 
\[
P^g(x_1, \ldots, x_n)=P(gx_1, \ldots, gx_n).
\]
Clearly, $P^g$ is a polynomial of the same degree as $P$.

{\bf Definition}
For each positive integer $r$, let $H_{n,r}(\F_q)$ denote the subspace of $A(\F_q)$  
consisting of all homogeneous
polynomials of degree $r$ in the $x_i$, along with the zero polynomial. 

As is well known, 
$H_{n,r}(\F_q)$ is a $GL(n,q)$-submodule of $A$ (in other words, $GL(n,q)$ maps the space of homogeneous polynomials of degree $r$ into itself).
Of course, these concepts hold for any field, not just $\F_q$, 
and we will work with corresponding spaces of polynomials
defined over $F$ in the next section. We note the dimension formula
\[
\dim_{\F_q}H_{n,r}(\F_q)=\binom{n+r-1}{r},
\]
where the expression on the right is the binomial coefficient. 

As before, let $E$ denote a splitting field for $L$ over $F$. Let $v_1$, \dots, $v_n \in E$ 
be an $\F_q$-basis for $V$. Consider the
mapping $\epsilon_r:H_{n,r}(\F_q)\longrightarrow E$ defined by
\[
\epsilon_rP(x_1, \ldots, x_n)=P(v_1, \ldots, v_n).
\]
Thus $\epsilon_r$ evaluates a homogeneous polynomial of degree $r$ on the basis of $V$ and hence determines an element of $E$. 
Note that $\epsilon_r$ is an $\F_q$-linear mapping. 
Let
$\epsilon_r(V)$ denote the $\F_q$-subspace of $E$ spanned by the image of $\epsilon_r$. 
As we shall show in the next lemma, this subspace
is independent of the choice of basis and thus the notation $\epsilon_r(V)$ makes sense. 

The following lemma is surely well known but we include a proof for definiteness.

\begin{lemma} \label{independent_of_basis}
Let $v_1$, \dots, $v_n$ and $w_1$, \dots, $w_n$ be two ordered bases of $V$ and let $\epsilon_r$ 
and $\epsilon_r'$ be the corresponding
maps from $H_{n,r}(\F_q)$ to $E$ determined by the bases. Then the $\F_q$-subspaces of $E$ spanned by the images of $\epsilon_r$ and $\epsilon_r'$
are identical. Thus we can speak unambiguously of the subspace $\epsilon_r(V)$ contained in $E$.
\end{lemma}

\begin{proof}
As we have two $\F_q$-bases of $V$, there is a unique $g\in GL(n,q)$ with
$gv_i=w_i$ for $1\leq i\leq n$. Given $P\in H_{n,r}(\F_q)$, we gave
\[
P^g(v_1, \ldots, v_n)=P(w_1, \ldots, w_n).
\]
Since $g$ is invertible, $P^g$ runs over $H_{n,r}(\F_q)$ as $P$ runs over $H_{n,r}(\F_q)$. 
It follows that the images of $\epsilon_r$
and $\epsilon_r'$ span the same subspace of $E$.
\end{proof}

The question we wish to raise is the following. Given a $q$-polynomial $L$ and space of roots $V$ of $L$, for what values
of $r$ is $\epsilon_r$ injective? Because, if $\epsilon_r$ is injective, then
we can identify $\epsilon_r(V)$ as the $r$-th symmetric power
of $V$, where
we recall that the $r$-th symmetric power of a vector space $V$ is the subspace 
of the symmetric algebra of $V$ consisting of all degree $r$ elements under the tensor product.
We will show later by a trivial argument (Theorem \ref{q+1_not_injective}) that if $F$ is the finite field $\F_q$,
$\epsilon_r$ is not injective if $r\geq q+1$ but it is injective if $r<q$. 
We note on the other hand that if $F$ is infinite,
the possibility exists that $\epsilon_r$ is injective for all $r$.
We shall see some examples of this.

\section{Projective Polynomials and Symmetric Powers}\label{proj_symm}

We recall the setup from earlier sections, except that we will assume $q=p$ in this section.

Let $F$ be a field of prime characteristic $p$. Let $L$ be a $p$-polynomial of $p$-degree 
$n$ in $F[x]$.
Let $E$ be a splitting field for $L$ over $F$ and let $V_L$ be the set of roots of $L$ in $E$. 
Let $G$ be the Galois group of $E$
over $F$. 
We assume that the roots of $L$ are distinct.
Let $M(x)=L(x)/x$.
Then it is easy to see that $M(x)$ is a monic polynomial in $x^{p-1}$,
say $M(x)=P(x^{p-1})$.
The polynomial $P(x)$ is monic of degree $(p^n-1)/(p-1)$ and is called the
projective polynomial associated to $L$.

\begin{lemma}\label{multi}
Let $L_P$ be the linearized polynomial of smallest degree that is divisible by $P$,
as in  Lemma \ref{mind2}.
Let $\alpha_1, \ldots , \alpha_n$ be a basis for $V_L$.
%Let $r \in \{1,2,3,\ldots ,p-1\}$.
Then the elements
$\alpha_1^{k_1}\alpha_2^{k_2}\cdots \alpha_n^{k_n}$ 
where $k_1+\cdots +k_n =p-1$ (and the $k_j$ are nonnegative integers) 
are roots of $L_P$.
\end{lemma}

\begin{proof}

Given notation above, 
let $\alpha_1, \ldots , \alpha_n$ be the roots of $L$ in $E$.
We know that if $\beta$ is any root of $L$ then $\beta^{p-1}$ is a root of $P$,
so the elements $(i_1\alpha_1 + i_2 \alpha_2 + \cdots + i_n \alpha_n)^{p-1}$ 
are roots of $P$, for any $i_1, \ldots ,i_n \in \F_p$.
We claim that each  $\alpha_1^{k_1}\alpha_2^{k_2}\cdots \alpha_n^{k_n}$ 
where $k_1+\cdots +k_n =p-1$ (and the $k_j$ are nonnegative integers) 
can be written as an $\F_p$-linear combination 
of the elements $(i_1\alpha_1 + i_2 \alpha_2 + \cdots + i_n \alpha_n)^{p-1}$,
and therefore they are roots of $L_P$.

By the Multinomial Theorem we write
\[
(i_1\alpha_1 + i_2 \alpha_2 + \cdots + i_n \alpha_n)^{p-1}= 
\sum_{k_1+\cdots +k_n=p-1} \binom{p-1}{k_1,\ldots,k_n} 
i_1^{k_1}\cdots i_n^{k_n} \alpha_1^{k_1}\cdots \alpha_n^{k_n}
\]
which gives us a system of  linear equations 
\[
\left(
\begin{array}{c}
\vdots \\
(i_1\alpha_1  + \cdots + i_n \alpha_n)^{p-1}\\
\vdots \\
\end{array}
\right)=
\left(
\begin{array}{ccc}
  \vdots& \vdots & \vdots\\
  \hdots   &i_1^{k_1}\cdots i_n^{k_n}  \binom{p-1}{k_1,\ldots,k_n}  &\hdots \\
   \vdots & \vdots & \vdots\\
\end{array}
\right)
\left(
\begin{array}{c}
\vdots \\
 \alpha_1^{k_1}\cdots \alpha_n^{k_n}\\
\vdots \\
\end{array}
\right)
\]
where the coefficient matrix has rows labelled by $(i_1,\ldots ,i_n)$
and columns labelled $(k_1,\ldots ,k_n)$.
In the column labelled $(k_1,\ldots ,k_n)$ every  term contains $ \binom{p-1}{k_1,\ldots,k_n}$
so this nonzero factor may be taken out. The remaining matrix is
\[
\left(
\begin{array}{ccc}
  \vdots& \vdots & \vdots\\
  \hdots   &(i_1^{k_1}\cdots i_n^{k_n})   &\hdots \\
   \vdots & \vdots & \vdots\\
\end{array}
\right)
\]
which is a non-square van der Monde matrix, having $p^n$ rows and 
$\binom{n+p-2}{p-1}$ columns.
Removing $p^n-\binom{n+p-2}{p-1}$ rows at the bottom  leaves a square van der Monde matrix,
which is invertible, and this proves the claim.
\end{proof}

%This argument proves that the evaluation map is injective, regardless of the Galois group of $L$.
We remark that
the same argument works when $p-1$ is replaced by any divisor $r$ of $p-1$,
i.e., where $P$ is defined by $L(x)/x=P(x^r)$.
An important point is that the multinomial coefficients
$ \binom{r}{k_1,\ldots,k_n}$ are nonzero when $1\le r \le p-1$.
This is one reason why we assume that $q=p$;
extending this lemma from $p$ to $q=p^a$ is not immediate.
A  generalization seems to be possible, it will involve the $p$-adic representation of 
divisors of $q-1$, and twisted tensor products using a Frobenius action.

\begin{cor}\label{space_roots_hom}
Let $F$ be a field of prime characteristic $p$. Let $L$ be a $p$-polynomial of $p$-degree 
$n$ in $F[x]$ with no repeated roots.
Let $E$ be a splitting field for $L$ over $F$ and let $V_L$ be the set of roots of $L$ in $E$. 
Let $r$ be a divisor of $p-1$.
Let $P(x)$ be defined by  $L(x)/x=P(x^{r})$.
Let $L_P$ be the linearized polynomial of smallest degree that is divisible by $P$,
as in  Lemma \ref{mind2}.
Then the space of roots of $L_P$ is a homomorphic image of the 
$r$-th symmetric power of $V_L$.
\end{cor}

\begin{proof}
From Lemma \ref{multi} the evaluation map $\epsilon_r$
maps surjectively, but perhaps not injectively,  into the space  of roots of $L_P$.
\end{proof}

{\bf Example:} let us choose a linearized polynomial $L$ of $p$-degree 2
such that the necessary hypotheses  hold.
Let $V_L$ be the space of roots of $L$, with basis $\alpha$ and $\beta$.
The $(p-1)$-th symmetric power of $V_L$ has dimension $p$.
%By the theorem of Doty, this is an irreducible module for $G_L$.
By Lemma \ref{multi} the space  of roots of $L_P$,
which is spanned by
$\alpha^i \beta^{p-1-i}$, $0\le i \le p-1$,
is a homomorphic image of the $(p-1)$-th symmetric power of $V_L$.
%The roots of $P$ are $(\alpha +i \beta)^{p-1}$, $i\in \F_p$, and $\beta^{p-1}$.

In any case that $\epsilon_r$ is injective, we may conclude that 
the space of roots of $L_P$ is \emph{isomorphic} to the 
$r$-th symmetric power of $V_L$.
For this reason, the next sections investigate the injectivity of $\epsilon_r$.

\section{Representation Theory and Modules}\label{RTM}

In this section we discuss one hypothesis that will guarantee the injectivity of $\epsilon_r$.

Recall the notation of the evaluation mapping $\epsilon_r$ from the space $H_{n,r}(\F_q)$
of homogeneous polynomials of degree $r$ into the splitting field $E$ of some $q$-polynomial $L$ of $q$-degree $n$ in $F[x]$. Similarly we 
let $\eta_r$ denote the evaluation mapping from $H_{n,r}(F) \longrightarrow E$.

Suppose that $L$ has no repeated roots and $G$ is the Galois group of $E$ over $F$. Then $G$  acts on the $\F_q$-space $H_{n,r}(\F_q)$ and on the $F$-space $H_{n,r}(F)$. (We only use the
fact that $G$ acts linearly on $V$ and on $F$.) It is easy to see that $\ker \epsilon_r$ is an $\F_q G$-submodule of $H_{n,r}(\F_q)$
and similarly, $\ker \eta_r$  is an $FG$-submodule of $H_{n,r}(F)$. If we know enough about the actions of $G$ on the two spaces of polynomials, we may be able to deduce something about the respective kernels.
In particular,  if $H_{n,r}(\F_q)$  is an irreducible module, then 
$\ker \epsilon_r$ is trivial.
This is one way to guarantee an injective $\epsilon_r$, the irreducibility of the module.
We state this as a Corollary.

\begin{cor}\label{space_roots_hom_irred}
Let $F$ be a field of prime characteristic $p$. Let $L$ be a $p$-polynomial of $p$-degree 
$n$ in $F[x]$ with no repeated roots.
Let $E$ be a splitting field for $L$ over $F$ and let $V_L$ be the set of roots of $L$ in $E$. 
Let $G$ be the Galois group of $E$ over $F$.
Let $r$ be a divisor of $p-1$.
Suppose that the $r$-th symmetric power of $V_L$ is an irreducible $FG$-module.
Let $P(x)$ be defined by  $L(x)/x=P(x^{r})$.
Let $L_P$ be the linearized polynomial of smallest degree that is divisible by $P$,
as in  Lemma \ref{mind2}.
Then the 
$r$-th symmetric power of $V_L$ is isomorphic to the space of roots of $L_P$.
\end{cor}

\begin{proof}
Since the kernel of $\epsilon_r$ is an $FG$-submodule of the 
$r$-th symmetric power of $V_L$, and this module is irreducible, the Corollary
follows from Corollary \ref{space_roots_hom}.
\end{proof}

In the cases that the Galois group is $GL(n,q)$ or $SL(n,q)$, 
it follows\footnote{using a theorem of Steinberg, which states that an irreducible module over
the algebraic closure of $\F_q$ remains irreducible upon restriction to $\F_q$}
from a theorem of Doty \cite{Doty}  that $H_{n,r}(\F_q)$ is an irreducible module
for $1\le r \le p-1$, or equivalently, that the $r$-th symmetric power
is an irreducible module for  $1\le r \le p-1$.

As a sample application of this Corollary, we take the example of 
a $p$-linearized polynomial over $\F_p(t)$ of $p$-degree $n$ with Galois group $GL(n,p)$.
This is the generic case, and an explicit example is
$L(x)=x^{p^n}+x^p+tx$ (due to Abhyankar).
Then $V_L$ has dimension $n$ over $\F_p$.
The $(p-1)$-th symmetric power of $V_L$ has dimension $\binom{n+p-2}{p-1}$.
We let $L_P$ be the linearized polynomial of smallest degree that is divisible by $P$,
as in  Lemma \ref{mind2}.
By Doty's theorem we know that the $(p-1)$-th symmetric power of $V_L$
is an irreducible $GL(n,p)$-module.
By Corollary \ref{space_roots_hom_irred} the
$(p-1)$-th symmetric power of $V_L$ is isomorphic to
 the space of roots of $L_P$.

As an example, let us take $n=2$ and $F=\F_p(t)$.
Choose a linearized polynomial $L$ of $p$-degree 2
having Galois group $G_L=SL(2,p)$ or $GL(2,p)$.
Let $V_L$ be the natural 2-dimensional module for $G_L$,
i.e.,  the space of roots of $L$.
Let $\alpha$ and $\beta$ be a basis.
The $(p-1)$-th symmetric power of $V_L$ has dimension $p$.
By the theorem of Doty, this is an irreducible module for $G_L$.
By Lemma \ref{multi}, and  Theorem  \ref{q-degree_2}, we can realize this module 
inside the splitting field of $L$ as the space  of roots of $L_P$,
which has a basis
$\alpha^i \beta^{p-1-i}$, $0\le i \le p-1$.
%The roots of $P$ are $(\alpha +i \beta)^{p-1}$, $i\in \F_p$, and $\beta^{p-1}$.

\section{Injectivity}\label{inj_sect}

We note the following simple principle relating to non-injectivity of $\epsilon_r$.

\begin{lemma} \label{generically_noninjective} 
Let $L$ be a $q$-polynomial in $F[x]$ and suppose that $L$ has no repeated roots in its splitting field
$E$ over $F$. Suppose that $\epsilon_r:H_{n,r}(\F_q)\to E$ 
is not injective. Then $\epsilon_t:H_{n,t}(\F_q)\to E$ is not injective for
all $t\geq r$.
\end{lemma}

\begin{proof}
Suppose that the homogeneous
polynomial $P$ of degree $r$ vanishes when evaluated on a given basis of the space of roots. Then if we set $Q=x_1^{t-r}P$, $Q$ is homogeneous
of degree $t$ and it also vanishes on the basis.
\end{proof}

%The basic idea of the lemma is that we can define an $\F_q$-linear evaluation mapping, $\epsilon$, say, from $A(\F_q)$ to $E$.
%If $\epsilon$ is not injective, let $I$ denote its kernel. $I$ is an ideal in $A(\F_q)$ and if we choose a nonzero polynomial
%of smallest degree $m$ in $I$

It seems reasonable when investigating the question posed above to restrict attention to the case that $L(x)/x$ is irreducible
in $F[x]$. Note then that this hypothesis ensures that $L$ has no repeated roots. We begin by examining what must be the easiest case, when $F=\F_q$. 

\subsection{Finite Fields}

\begin{lemma} \label{special_basis}
Let $L$ be a $q$-polynomial in $\F_q[x]$ of $q$-degree $n$, such that $L(x)/x$ is irreducible in $\F_q[x]$. Let
$\alpha\neq 0$ be a root of $L$ in a splitting field over $\F_q$. Then $\alpha$, $\alpha^q$, \dots, $\alpha^{q^{n-1}}$ are linearly
independent over $\F_q$ and are a basis for the space of roots of $L$.
\end{lemma}

\begin{proof}
We note that $\alpha^{q^i}$ is a root of $L$ for all $i$, because $L$ has coefficients in $\F_q$. 
Suppose that we have a linear dependence relation
\[
\lambda_0\alpha+\lambda_1\alpha^q+\cdots+ \lambda_{n-1}\alpha^{q^{n-1}},
\]
where the $\lambda_i$ are in $\F_q$. Then, unless all the $\lambda_i$ are zero, $\alpha$ is the root of a nonzero polynomial
of degree at most $q^{n-1}$ over $\F_q$, and this contradicts the assumption that the minimal polynomial of $\alpha$ has degree
$q^n-1$. It follows that the $\lambda_i$ are all zero and we have proved that the $n$ powers of $\alpha$ are a basis of the space
of roots.
\end{proof}

The use of this special basis of $V$ enables us to prove that $\epsilon_r$ is a monomorphism for many values
of $r$. 

\begin{thm} \label{monomorphism_if_less_than_q}
Let $L$ be a $q$-polynomial in $\F_q[x]$ of $q$-degree $n$, such that $L(x)/x$ is irreducible in $\F_q[x]$. 
Let $V$ be the space of roots of $L$ in a splitting field $E$ of $L$ over $\F_q$ 
$(E$ is isomorphic to $\F_{q^d}$, where $d=q^n-1)$. 
Then $\epsilon_r: H_{n,r}(\F_q) \longrightarrow E$ is injective for $1\leq r\leq q-1$.
\end{thm}

\begin{proof}
Let $\alpha$ be a nonzero root of $L$ in $E$. We use the basis of $V$ described in Lemma \ref{special_basis} to study $\epsilon_r$.
We find that $\epsilon_r(V)$ consists of $\F_q$-linear combinations of powers $\alpha^t$, where we have
\[
t=r_0+r_1q+\cdots +r_{n-1}q^{n-1},
\]
and the $r_i$ are nonnegative integers that satisfy $r_0+r_1+\cdots +r_{n-1}=r$.

Now given the hypothesis that $r\leq q-1$, we have
\[
r_0+r_1q+\cdots +r_{n-1}q^{n-1}\leq (q-1)q^{n-1} <q^n-1.
\]
Thus the existence of any nontrivial dependence relation among the powers of $\alpha$ occurring in $\epsilon_r(V)$ implies
that $\alpha$ is a root of a nonzero polynomial in $\F_q[x]$ of degree less than $q^n-1$. This is a contradiction and we have
established the desired result.
\end{proof}

As we shall see later, it easy to show that $\epsilon_q$ is also injective but $\epsilon_{q+1}$ is not. The interest of restricting
to values of $r$ at most $q-1$ is suggested by our next result, where 
we give a more precise description of the way in which the subspaces $\epsilon_r(V)$ are embedded in $E$.

\begin{cor} \label{direct_sum}
Using the notation and hypotheses of Theorem \ref{monomorphism_if_less_than_q}, and taking 
$\epsilon_0(V)$ to be the one-dimensional
subspace spanned by $1$,
the splitting field $E$ contains the direct sum
\[
\epsilon_0(V)\oplus\epsilon_1(V)\oplus\cdots \oplus \epsilon_{q-1}(V)
\]
of dimension
\[
\frac{q(q+1)\cdots (q+n-1)}{n!}.
\]
\end{cor}

\begin{proof}
We have noted that $\epsilon_r(V)$ is spanned by powers of $\alpha$ where the exponents of $\alpha$ have the form
\[
t=r_0+r_1q+\cdots +r_{n-1}q^{n-1},
\]
and the $r_i$ are nonnegative integers that satisfy $r_0+r_1+\cdots +r_{n-1}=r$. The integer $t$ is expressed as a $q$-adic integer, which we will say has weight $r$. The expression is unique: its representation as a $q$-adic integer has exactly one weight. The number of such integers of weight $r$ is $\dim H_{n,r}(\F_q)$. 

If the sum of the subspaces $\epsilon_i(V)$ is not direct, we must have a dependence of the form
\[
\lambda_0v_0+\lambda_1v_1+ \cdots +\lambda_{q-1}v_{q-1},
\]
where the $\lambda_i$ are in $\F_q$ and the $v_i$ in $\epsilon_i(V)$. Each $v_i$ is an $\F_q$-combination of powers of $\alpha$, where the exponents have weight $i$. Thus the dependence involves $\dim H_{n,0}+\dim H_{n,1}+\cdots+ \dim H_{n,q-1}$ different powers of $\alpha$.
Since the powers $\alpha^i$, $0\leq i<q^n-1$, are linearly independent, this is clearly impossible and we deduce that the sum is direct.

Finally, the dimension of the direct sum is
\[
\sum_{i=0}^{q-1}\binom{n+i-1}{i}=\frac{q(q+1)\cdots (q+n-1)}{n!}.
\]
\end{proof}

As we remarked earlier, the following result holds in the finite field case.

\begin{thm} \label{q+1_not_injective}
Assume the hypotheses of Theorem \ref{monomorphism_if_less_than_q} and suppose that $n\geq 2$. Then the mapping
$\epsilon_{r}$ is not injective for $r\geq q+1$.
\end{thm}

\begin{proof}
We use the basis $v_1=\alpha$, $v_2=\alpha^q$, \dots, $v_n=\alpha^{q^{n-1}}$ of the space of roots of $L$ and assume first 
that $n\geq 3$. Lemma
\ref{generically_noninjective} shows that it suffices to prove the result when $r=q+1$.
Consider the homogeneous polynomial 
\[
P(x_1, \ldots, x_n)=x_2^{q+1}-x_1^qx_3.
\]
of degree $q+1$.
It is clear that when $P$ is evaluated on the given basis, it vanishes. 

Now suppose that $n=2$. In this case, we can assume that $L=x^{q^2}+ax^q+b$ for suitable $a$ and $b$ in $\F_q$. Consider the polynomial
\[
P(x_1,x_2)=x_2^{q+1}+ax_1^qx_2+bx_1^{q+1}
\]
in $\F_q[x]$. We have
\[
P(\alpha, \alpha^q)=\alpha^{q(q+1)}+a\alpha^{2q}+b\alpha^{q+1}=\alpha^q(\alpha^{q^2}+a\alpha^{q}+b\alpha)=0.
\]
Thus $P$ vanishes on the basis and we deduce that $\epsilon_{q+1}$ is not injective in this case also.
\end{proof}

We shall describe next one circumstance where we cannot
expect to find any analogue of Theorem \ref{monomorphism_if_less_than_q}. This relates to the subject of
$q^s$-polynomials. 

\begin{thm} \label{q^s-polynomial}
Suppose that $s>1$ is an integer and the field $\mathbb{F}_{q^s}$ is not contained in $F$. Let $L$ be a $q^s$-polynomial in $F[x]$ with no repeated roots in its splitting field $E$ over $F$. Then the mapping $\epsilon_s$ into $E$ is not injective.
\end{thm}

\begin{proof}
Let $\alpha$ be a nonzero root of $L$ in $E$ and let $\omega$ be an element of $\mathbb{F}_{q^s}$ not in $\F_q$ (recall that
$\mathbb{F}_{q^s}$ is contained in $E$ by Lemma  \ref{semi_linear_automorphism}). 
Let $\beta=\omega\alpha$. Then $\alpha$ and $\beta$ are roots of $L$ that are linearly
independent over $\F_q$. 

The powers $1, \alpha, \alpha^2, \ldots, \alpha^s$ are linearly dependent over $\F_q$, since $\mathbb{F}_{q^s}$ has dimension $s$
over $\F_q$. It follows that $\alpha^s$, $\alpha^{s-1}\beta$, \dots, $\beta^s$ are linearly dependent over $\F_q$, say
\[
\lambda_0 \alpha^s+\lambda_1\alpha^{s-1}\beta+\cdots +\lambda_s\beta^s=0,
\]
where the $\lambda_i\in\F_q$ and are not all zero.

Consider the homogeneous polynomial $P$ of degree $s$ in the variables $x_1$, \dots, $x_n$ (where $n$ is the $q$-degree of $L$) given
by
\[
P(x_1, \ldots, x_n)=\lambda_0 x_1^s+\lambda_1 x_1^{s-1}x_2+\cdots +\lambda_s x_2^s,
\]
all other terms being 0. We extend the linear independent roots $\alpha$ and $\beta$ to a basis $v_1=\alpha$, $v_2=\beta$, \dots, $v_n$
of the space of roots of $L$. Then we find that
\[
\epsilon_sP(x_1, \ldots, x_n)=P(v_1, \ldots, v_n)=0.
\]
Since $P$ is a nonzero polynomial in $H_{n,s}(\F_q)$ that is in the kernel of $\epsilon_s$, we see that $\epsilon_s$ is not injective.
\end{proof}

It may be observed that in the previous theorem, 
the space  of roots of $L$ is more naturally a vector space over $\mathbb{F}_{q^s}$ than one over
$\F_q$, and we might expect it to be easier to obtain dependencies when we work over $\F_q$, rather than $\mathbb{F}_{q^s}$. Nonetheless, we feel that the theorem serves to indicate that caution is necessary when we try to generalize Theorem \ref{monomorphism_if_less_than_q} to other situations.

In this connection the following theorem related to Theorem \ref{q^s-polynomial} is of interest, although the proof we use is unrelated
to the ideas developed so far.

\begin{thm} \label{reducible_if_q^s_polynomial}
Let $L$ be a $q$-polynomial in $\F_q[x]$ such that $L(x)/x$ is irreducible over $\F_q$. Then $L$ is not a $q^s$-polynomial for
any $s>1$.
\end{thm}

\begin{proof}
Let the $q$-degree of $L$ be $n$ and let
\[
L=a_nx^{q^n}+a_{n-1}x^{q^{n-1}}+\cdots +a_0x,
\]
where $a_n\neq 0$. The polynomial $\ell(x)$ in $\F_q[x]$ defined by
\[
\ell(x)=a_nx^n+a_{n-1}x^{n-1}+\cdots +a_0
\]
is called the conventional $q$-associate of $L$.

Given that $L(x)/x$ is irreducible, Theorem 3.63 of \cite{LN} 
implies that $\ell(x)$ is irreducible and any root $\alpha$ of $\ell(x)$ in
$\mathbb{F}_{q^n}$ has multiplicative order $q^n-1$. Suppose now that $L$ is a $q^s$-polynomial, where $s>1$. Then $n=ms$ for some integer
$m$ and $\ell$ must be a polynomial in $x^s$, say $\ell(x)=g(x^s)$ where $g$ has degree $m$. Now we have $\ell(\alpha)=0=g(\alpha^s)$.
This shows that $\alpha^s$ is a root of $g$. 

Since $\ell$ is irreducible of degree $m$, its roots have order dividing $q^m-1$. Thus $\alpha^{s(q^m-1)}=1$. On the other hand,
Theorem 3.63 of \cite{LN} 
shows that $\alpha$ has order $q^{ms}-1$. Thus $q^{ms}-1$ divides $s(q^m-1)$. This is clearly impossible if $s>1$.
Thus $L$ is not a $q^s$-polynomial.
\end{proof}

This theorem  only applies to finite fields.

\subsection{Arbitrary Fields, Dimension 2}

We now move to the case where the field of coefficients $F$ is an arbitrary 
field of characteristic $p$.
We wish to examine the case of two linearly independent roots 
$\alpha$ and $\beta$ of a $q$-polynomial.
We present two theorems, one concerning linear dependence of
$\alpha^d$, $\alpha^{d-1}\beta$, \dots, $\beta^d$  over $\F_q$,
and the other concerning linear dependence over $F$.
These results will be applied in the next section to 
give us more special cases in our investigations of the injectivity of
the evaluation maps.

\begin{thm} \label{estimate_of_m}
 Let $L$ be a monic $q$-polynomial of $q$-degree $n$ in $F[x]$ with the following properties:
 
 \begin{enumerate}
 \item $L(x)/x$ is irreducible over $F$.
 \item $L$ is not a $q^s$-polynomial for any integer $s>1$.
 \item $L$ has the form
 \[
 x^{q^n}+a_{n-k}x^{q^{n-k}}+\cdots +a_0x,
 \]
 where $a_{n-k}\neq 0$, i.e. either $k=1$ and $a_{n-1}\not=0$, or  
 $k>1$ and $a_{n-1}=\cdots =a_{n-k+1}=0$.
 \item $V$ is the space of roots of $L$ in a splitting field $E$ of $L$ over $F$.
 \end{enumerate}
 
 Let $\alpha$ and $\beta$ be elements of $V$ linearly independent over $\F_q$. Suppose that there exists a positive integer $d$
 such that $\alpha^d$, $\alpha^{d-1}\beta$, \dots, $\beta^d$ are linearly dependent over $\F_q$
 (such $d$ may not exist). 
 Then $d\geq q^k+1$.
 \end{thm}
 
 \begin{proof}
 Let $m$ be the smallest such $d$ and let $\lambda_0$, \dots, $\lambda_m$ be elements of $\F_q$, not all 0, with
 \[
 \lambda_0\alpha^m+\lambda_1 \alpha^{m-1}\beta+\cdots +\lambda_m\beta^m=0.
 \]
 We claim that $\lambda_0\neq 0$. For if $\lambda_0=0$, we may divide the resulting equality by $\beta$ to obtain
 \[
 \lambda_1\alpha^{m-1}+\cdots +\lambda_m\beta^{m-1}=0,
 \]
 and this contradicts the minimality of $m$. Thus $\lambda_0\neq 0$ and likewise $\lambda_m\neq 0$. Note also that $m>1$, since
 we are assuming that $\alpha$ and $\beta$ are linearly independent over $\F_q$.
 
 We now see that if we set $\gamma=\alpha/\beta$, $\gamma$ is a root of a polynomial of degree $m$ over $\F_q$. The minimality of
 $m$ again shows that this polynomial is irreducible over $\F_q$. We deduce that $\gamma\in \mathbb{F}_{q^m}$ and hence
 $\gamma^{q^m}=\gamma$.
 
 Since $\alpha=\gamma \beta $ and $\beta$ are roots of $L$, we have
 \[
 \beta^{q^n}+a_{n-k}\beta^{q^{n-k}}+\cdots +a_0 \beta=0
 \]
 and likewise
 \[
 \gamma^{q^n}\beta^{q^n}+\gamma^{q^{n-k}}a_{n-k}\beta^{q^{n-k}}+\cdots +\gamma a_0 \beta=0.
 \]
 We divide by $\gamma^{q^n}$ and subtract the resulting two monic equations involving powers of $\beta$ to obtain
 \[
 a_{n-k}(\gamma^{(q^{n-k}-q^n)}-1)\beta^{q^{n-k}}+\cdots +a_0(\gamma^{1-q^n}-1)\beta=0.
 \]
 We see that $\beta$ is a root of a $q$-polynomial of $q$-degree at most $n-k$ over the field $F(\gamma)$. 
 
 Suppose if possible that this polynomial is zero. It follows that 
 \[
 a_{n-i}(\gamma^{(q^{n-i}-q^n)}-1)=0
 \]
 for $k\leq i\leq n$. It is certainly true that $a_0\neq 0$ under our hypothesis that $L(x)/x$ is irreducible. 
 Then the $i=n$ case yields  $\gamma^{q^n}=\gamma$. 
 Thus since we know that $\gamma\in\mathbb{F}_{q^m}$ and no smaller field $\mathbb{F}_{q^t}$, we deduce
 that $m$ divides $n$.
  
  Consider now an equation
   \[
 a_{n-i}(\gamma^{(q^{n-i}-q^n)}-1)=0
 \]
 and suppose that $a_{n-i}\neq 0$. Then we have
 \[
 \gamma^{q^n}=\gamma^{q^{n-i}}.
 \]
 The argument just applied above shows that $m$ divides $n-i$. This implies that $L$ is a $q^m$-polynomial, where $m>1$, contrary
 to hypothesis. We deduce that $\beta$ is indeed a root of a nonzero $q$-polynomial of $q$-degree at most $n-k$ over $F(\gamma)$.
 
 We obtain the inequality
 \[
 [F(\gamma,\beta):F(\gamma)]\leq q^{n-k}-1.
 \]
 Since $\F_q$ is assumed to be a subfield of $F$, and since the minimal polynomial of $\gamma$ over $\F_q$ has degree $m$, 
 we have
 \[
 [F(\gamma):F]\leq m.
 \]
 Thus we obtain
 \[
 [F(\gamma,\beta):F]\leq m(q^{n-k}-1).
 \]
 
 Since $F(\beta)$ is a subfield of $F(\gamma, \beta)$, we obtain the inequality
 \[
  [F(\gamma,\beta):F]=[F(\gamma,\beta):F(\beta)][F(\beta):F]=[F(\gamma,\beta):F(\beta)](q^n-1).\
  \]
  Hence the inequality
  \[
  [F(\gamma,\beta):F(\beta)](q^n-1)\leq m(q^{n-k}-1)
  \]
  holds. We deduce that
  \[
  m\geq (q^n-1)/(q^{n-k}-1).
  \]
  Since $m$ is an integer, we obtain that $m\geq q^k+1$. Thus the $d$ in the statement of the theorem is at least $q^k+1$.
 \end{proof}

So far we studied the evaluation mapping $\epsilon_r$ from the space $H_{n,r}(\F_q)$
of homogeneous polynomials of degree $r$ into the splitting field $E$ of some $q$-polynomial $L$ of $q$-degree $n$ in $F[x]$.
It seems reasonable to use the field $\F_q$ for coefficients
because the space of roots of $L$ is a vector space over $\F_q$.
However, we can also study evaluations when we replace $H_{n,r}(\F_q)$ by $H_{n,r}(F)$ but still evaluate on a basis of $V$. We shall
let $\eta_r$ denote the evaluation mapping from $H_{n,r}(F) \longrightarrow E$. 
The image of $\eta_r$ is an $F$-subspace of $E$. Since the dimension
of $H_{n,r}(F)$ over $F$ increases monotonically as $r$ increases, whereas $E$ has finite dimension over $F$, $\eta_r$ is not injective
for almost all $r$. This suggests the problem of finding or estimating the smallest $r$ for 
which $\eta_r$ is not injective.

 The next Theorem is similar to Theorem \ref{estimate_of_m}, but is not the same,
 because the next Theorem assumes
 that $\alpha^d$, $\alpha^{d-1}\beta$, \dots, $\beta^d$ are linearly dependent over $F$,
 whereas the previous Theorem assumed linear dependence over $\F_q$.

\begin{thm} \label{estimate_of_m_when_using_F}
 Let $L$ be a monic $q$-polynomial of $q$-degree $n$ in $F[x]$ with the following properties:
 
 \begin{enumerate}
 \item $L(x)/x$ is irreducible over $F$.
 \item $L$ is not a $q^s$-polynomial for any integer $s>1$.
 \item $L$ has the form
 \[
 x^{q^n}+a_{n-k}x^{q^{n-k}}+\cdots +a_0x,
 \]
 where $a_{n-k}\neq 0$ (thus $a_{n-1}=\cdots a_{n-k+1}=0$ if $k>1$).
 \item $V$ is the space of roots of $L$ in a splitting field $E$ of $L$ over $F$.
 \end{enumerate}
 
 Let $\alpha$ and $\beta$ be elements of $V$ linearly independent over $\F_q$. Let $d$ be a positive integer 
 such that $\alpha^d$, $\alpha^{d-1}\beta$, \dots, $\beta^d$ are linearly dependent over $F$
 (such  $d$ will certainly exist).
 Then $d\geq q^k+1$.
 \end{thm}
 
 \begin{proof}
 The proof is very similar to that of Theorem \ref{estimate_of_m}, so we sketch the details.
 
 Let $m$ be the smallest  $d$ such that
 $\alpha^d$, $\alpha^{d-1}\beta$, \dots, $\beta^d$ are linearly dependent over $F$.
 We set $\gamma=\alpha/\beta$ and show that $\gamma$ is a root of an irreducible polynomial of degree $m$ over $F$, as in the proof of Theorem \ref{estimate_of_m}.
 Then we obtain that 
 \[
 a_{n-k}(\gamma^{(q^{n-k}-q^n)}-1)\beta^{q^{n-k}}+\cdots +a_0(\gamma^{1-q^n}-1)\beta=0.
 \]
 This shows that $\beta$ is a root of a $q$-polynomial of $q$-degree at most $n-k$ over the field $F(\gamma)$. 
 
 We consider the case that this polynomial is zero. It follows that 
 \[
 a_{n-i}(\gamma^{(q^{n-i}-q^n)}-1)=0
 \]
 for $k\leq i\leq n$. It is certainly true that $a_0\neq 0$ under our hypothesis that $L(x)/x$ is irreducible. We must then
 have $\gamma^{q^n}=\gamma$. This implies that $\gamma\in \mathbb{F}_{q^n}$.  Suppose that
 $\gamma\in\mathbb{F}_{q^m}$ and no smaller field $\mathbb{F}_{q^t}$, where $t<m$.  Then $m$ 
 divides $n$. Note that $m>1$, since $\alpha$ and $\beta$ are linearly independent over $\F_q$
  
  Consider now an equation
   \[
 a_{n-i}(\gamma^{(q^{n-i}-q^n)}-1)=0
 \]
 and suppose that $a_{n-i}\neq 0$. Then we have
 \[
 \gamma^{q^n}=\gamma^{q^{n-i}}.
 \]
 It must then be the case that $m$ divides $n-i$. This implies that $L$ is a $q^m$-polynomial and hence $m=1$, which
 we know not to be the case. It follows that $\beta$ is indeed a root of a nonzero $q$-polynomial of $q$-degree at most $n-k$ over $F(\gamma)$.
 
 We obtain the inequality
 \[
 [F(\gamma,\beta):F(\gamma)]\leq q^{n-k}-1.
 \]
 We also have 
 \[
 [F(\gamma):F]=m,
 \]
 since the minimal polynomial of $\gamma$ over $F$ has degree $m$.
 The rest of the proof follows as before.
 \end{proof}

 \subsection{$q$-Degree 2}\label{qdeg2}

 As we shall show, the  theorems in the previous section
 may be applied effectively to investigate the space of roots 
 of a $q$-polynomial of $q$-degree 2 but
 we need to
 make a hypothesis concerning the field $F$, 
 for without some assumption, the results we have in mind are not necessarily true.
 
 We are always assuming that $F$ contains $\F_q$.
 Our new hypothesis is that the largest finite subfield of $F$  is $\F_q$ itself. 
 This means that the only elements
 of $F$ that are algebraic over $\F_q$ are the elements of $\F_q$, or in other words, 
 $F$ is a regular extension of $\F_q$. 
A standard example of such a regular field $F$ is the function field $\F_q(t)$ in a single variable $t$ over $\F_q$.

 In the regular extension case, the following
 result is true. 
 
 \begin{lemma} \label{irreducible_in_larger_field}
 Suppose that the field is a regular extension of $\F_q$. Let $f$ be an irreducible polynomial in $\F_q[x]$. Then
 considered as a polynomial in $F[x]$, $f$ remains irreducible.
 \end{lemma}
 
 A proof may be found in  \cite{L}, Lemma 4.10, p.366.

 We recall here that
  if $L$ is a $q$-polynomial of $q$-degree 2 with no repeated roots, its Galois group is a subgroup of $GL(2,q)$, by Lemma \ref{vector_space}.
 The group $GL(2,q)$ contains a normal subgroup $SL(2,q)$, the special linear group, consisting of the elements of determinant 1.
 It is well known that if $q$ is a power of 2 greater than 2, then $SL(2,q)$ is a simple group. However, $SL(2,2)$ is isomorphic to the symmetric
 group $S_3$ and is anomalous as it has a normal subgroup of index 2.
 
 If $q$ is odd and greater than 3, $SL(2,q)$ contains no proper normal subgroups with abelian quotient (in other words, the group is perfect).
 $SL(2,3)$ contains a normal subgroup of order 8 with cyclic quotient of order 3.
 
 We may now proceed to our main theorem relating to $q$-polynomials of $q$-degree 2.
 
 \begin{thm} \label{q-degree_2}
 Let $L$ be a $q$-polynomial of $q$-degree $2$ in $F[x]$. Suppose that the following hold.
 \begin{enumerate}
 \item $L(x)/x$ is irreducible over $F$.
 \item $q>2$.
 \item $F$ is a regular extension of $\F_q$.
 \item $E$ is the splitting field for $L$ over $F$.
 \item The Galois group $G$ of $E$ over $F$ contains $SL(2,q)$.
 \end{enumerate} 
 
 Then the evaluation mapping $\epsilon_r$ from the space $H_{2,r}(\F_q)$ of homogeneous polynomials of degree $r$ in two variables into $E$ is injective
 for all $r\geq 1$.
\end{thm}

\begin{proof}
We first show that $L$ is not a $q^2$-polynomial. For if $L$ is such a polynomial, the space $V$ of roots is a one-dimensional
vector space over $\mathbb{F}_{q^2}$ and the Galois group acts semilinearly on the vector space 
(see Lemma \ref{semi_linear_automorphism}).
Since the Galois group of $\mathbb{F}_{q^2}$ over $\F_q$ has order 2, $G$ has order dividing $2(q^2-1)$. Since we are assuming that
$G$ contains $SL(2,q)$ whose order is $q(q^2-1)$, we must have the inequality $q(q^2-1)\leq 2(q^2-1)$ and hence $q=2$, a possibility
we have excluded. It follows that $L$ is not a $q^2$-polynomial under our hypothesis.

Let $\alpha$ and $\beta$ be elements of the space of roots of $L$ that are linearly independent over $\F_q$. Let
$\epsilon_m$ be the evaluation mapping from $H_{2,m}(\F_q)$.
% defined by
%\[
%\epsilon_mP=P(\alpha,\beta)\in E.
%\]
Suppose that $\epsilon_m
$ is not injective, and $m$ is chosen minimal with this property. Then $\alpha^m$, $\alpha^{m-1}\beta$, \dots, $\beta^m$ are linearly dependent over $\F_q$ and it follows from Theorem \ref{estimate_of_m} that $m\geq q+1$. 

We have proved in Theorem \ref{estimate_of_m} that $\gamma=\alpha/\beta$ is a root of an irreducible polynomial $f$ of degree $m$ in
$\F_q[x]$. Since $F$ is regular over $\F_q$, Lemma \ref{irreducible_in_larger_field} implies that $f$
is also irreducible over $F$ and hence $[F(\gamma):F]=m$. Now $F(\gamma)$ is a normal subfield of $E$, since it is the splitting field
for $f$ over $F$. The Galois group $G$ thus maps $F(\gamma)$ into itself and induces the Galois group of $F(\gamma)$ over $F$ by its action.
The Galois group is cyclic of order $m$, since it is isomorphic to the Galois group of $\F_q(\gamma)$ over $\F_q$. Let $H$ be the subgroup
of $G$ that acts trivially on $F(\gamma)$. Elementary Galois theory tells us that $H$ is normal in $G$ and $G/H$ is cyclic of order $m$.

Let $S$ be a subgroup of $G$ isomorphic to $SL(2,q)$. Then $SH$ is a subgroup of $G$ and the quotient
\[
SH/H\cong S/S\cap H
\]
is a subgroup of $G/H$ and is thus cyclic. Since we remarked before the proof that $SL(2,q)$ is perfect for $q>3$, we deduce that
$S$ is contained in $H$ when $q>3$.

To finish the proof, assume $q>3$ and consider $G$ as a subgroup of $GL(2,q)$. 
Then $|G:H|=m$ is a divisor of $|GL(2,q):H|$, and since $H$ contains
$SL(2,q)$, $m$ is then a divisor of $|GL(2,q):SL(2,q)|=q-1$. This is impossible, as we already know that
$m\geq q+1$. In the case that $q=3$, our assumption is that $G$ is either $SL(2,3)$ or $GL(2,3)$. Since neither of these groups 
has an abelian
quotient of order greater than 3, the theorem holds in this case also.
\end{proof}

We remark that the theorem also holds in the excluded case $q=2$ provided that $L$ is not a 4-polynomial.

 The next theorem is similar to Theorem  \ref{q-degree_2}
 however it does not assume that $F$ is a regular extension, and 
 it does not assume that the Galois group contains $SL(2,q)$.
 The conclusion this time is about injectivity of evaluation maps $\eta_r$ over $F$, 
 which we know will fail when $r$ is sufficiently large,
 whereas Theorem  \ref{q-degree_2} is  about the evaluation maps $\epsilon_r$ over $\F_q$,
 which can be injective for all $r$.
 
 \begin{thm} \label{q-degree_2_F_version}
 Let $L$ be a $q$-polynomial of $q$-degree $2$ in $F[x]$. Suppose that the following hold.
 \begin{enumerate}
 \item $L(x)/x$ is irreducible over $F$.
 \item $L$ is not of the form $x^{q^2}+bx$ for some nonzero $b\in F$.
 \item $E$ is the splitting field for $L$ over $F$.
 \end{enumerate} 
 
 Then the evaluation mapping $\eta_r$ from the space $H_{2,r}(F)$ into $E$ is injective if $r<[E:F]/(q-1)$. 
\end{thm}

\begin{proof}
Let $\alpha$ and $\beta$ be elements of the space of roots of $L$ that are linearly independent over $\F_q$. We claim that
$E=F(\alpha,\beta)$. This follows from the fact that $E$ is generated over $F$ by the roots of $L$ and these roots are $\F_q$-combinations
of $\alpha$ and $\beta$. Likewise, if we set $\gamma=\alpha/\beta$, we have $E=F(\alpha, \gamma)$.

Let $m$ be the smallest positive integer such that 
the evaluation mapping $\eta_m$  
is not injective. Then, as we showed in the proof of Theorem \ref{estimate_of_m_when_using_F}, $[F(\gamma):F]=m$.
Furthermore, the same proof shows that $[F(\alpha,\gamma):F(\gamma)]\leq q-1$.

We now assemble the parts to show that
\[
[E:F]=[F(\alpha,\gamma):F]=[F(\alpha,\gamma):F(\gamma)][F(\gamma):F]\leq m(q-1)
\]
and this yields the desired inequality for $m$.
\end{proof}

Based on the results so far, we make the following conjecture.

\begin{conj}\label{c1}
Let $L=L(x)$ be a $q$-polynomial in $F[x]$ of $q$-degree $n$, 
where $F$ is a field of characteristic $p$ that contains $\F_q$.
Assume that $L$ is not a $q^s$-polynomial for $s>1$.
Assume that $L(x)/x$ is irreducible over $F$.\\
%Let $E$ be a splitting field for $L$. Let $V\subseteq E$ be the $\F_q$-span of the roots of $L$.
If $1\le r \le q-1$ then $\epsilon_r$ is injective.
\end{conj}

\end{document}